\documentclass[12pt,reqno,table]{amsart}

\usepackage{amssymb,amsfonts,latexsym,amstext,epsfig,color,graphicx,tikz,graphics,float,fp,amsthm,cite,caption,mathtools,hyphenat,amsmath,systeme, lscape,rotating,subcaption,setspace,hhline}
\usepackage[left=3.5cm,top=3.5cm,right=3.5cm,bottom=3.5cm]{geometry}
\usepackage{makecell}
\usetikzlibrary{decorations.pathreplacing,calligraphy}
\newcolumntype{?}{!{\vrule width 1pt}}
\usepackage{array}

\newtheorem{theorem}{Theorem}[section]
\newtheorem{lemma}[theorem]{Lemma}
\newtheorem{corollary}[theorem]{Corollary}
\theoremstyle{definition}

\theoremstyle{conjecture}

\theoremstyle{example}

\newcommand{\loc}[1]{\zeta(#1)}

\usetikzlibrary{decorations.pathreplacing}
\usetikzlibrary{decorations.pathmorphing}
\usetikzlibrary{positioning}
\usetikzlibrary{snakes}

\begin{document}
\title{Pursuit-evasion games on latin square graphs}
\author[S.\ Ahirwar]{Shreya Ahirwar}
\author[A.\ Bonato]{Anthony Bonato}
\author[L.\ Gittins]{Leanna Gittins}
\author[A.\ Huang]{Alice Huang}
\author[T.G.\ Marbach]{Trent G.\ Marbach}
\author[T.\ Zaidman]{Tomer Zaidman}
\address[A1]{Department of Mathematics \& Statistics, Mount Holyoke College, South Hadley, U.S.A.}
\address[A2, A5]{Department of Mathematics, Ryerson University, Toronto, Canada,  M5B 2K3.}
\address[A3]{Department of Mathematics \& Statistics, McMaster University, Hamilton, Canada, L8S 4K1}
\address[A4, A6]{Department of Mathematics, University of Toronto, Toronto, Canada,  M5S 2E4.}
\email[A1]{(A1) ahirw22s@mtholyoke.edu}
\email[A2]{(A2) abonato@ryerson.ca}
\email[A3]{(A3) leanna.gittins@gmail.com}
\email[A4]{(A4) alicew.huang@mail.utoronto.ca}
\email[A5]{(A5) trent.marbach@gmail.com}
\email[A6]{(A6) tomer.zaidman@mail.utoronto.ca}

\begin{abstract}
We investigate various pursuit-evasion parameters on latin square graphs, including the cop number, metric dimension, and localization number. The cop number of latin square graphs is studied, and for $k$-MOLS$(n),$ bounds for the cop number are given. If $n>(k+1)^2,$ then the cop number is shown to be $k+2.$ Lower and upper bounds are provided for the metric dimension and localization number of latin square graphs. The metric dimension of back-circulant latin squares shows that the lower bound is close to tight. Recent results on covers and partial transversals of latin squares provide the upper bound of $n+O\left(\frac{\log{n}}{\log{\log{n}}}\right)$ on the localization number of a latin square graph of order $n.$
\end{abstract}


\maketitle

\section{Introduction}

Pursuit-evasion games, including the well-known game of Cops and Robbers and the localization game, are combinatorial models for detecting or neutralizing an
adversary's activity on a graph. In such models, pursuers attempt to capture an evader loose on the vertices of a graph. How the players move and the rules
of capture depend on which variant is studied. Such games are motivated by foundational topics in computer science, discrete mathematics, and
artificial intelligence, such as robotics and network security.  For surveys of pursuit-evasion games, see~\cite{bp,by,nisses}, and see \cite{BN} for more background on Cops and Robbers.

In Cops and Robbers, the pursuers are \emph{cops} and the evader is the \emph{robber}. Both players move on vertices. The cops move first, followed by the robber; the players then alternate moves. The robber is visible, and players move to adjacent vertices or remain on their current vertex. The cops win if, after a finite number of rounds, they can land on the vertex of the robber; otherwise, the robber wins. The least number of cops needed to guarantee that the robber is captured on a graph $G$ is the \emph{cop number} of $G,$ denoted by $c(G).$ Note that $c(G)$ is well-defined, as $c(G) \le \gamma(G),$ where $\gamma(G)$ is the domination number of $G.$ For more background on the cop number of a graph, see \cite{BN}.

In the \emph{localization game}, the robber moves first and is invisible to the cops during gameplay. As in Cops and Robbers, the robber occupies vertices and moves between vertices along edges. On their turn, the cops may move to any vertex of the graph. After each move, the cops occupy a set of vertices $u_1, u_2, \dots , u_k$ and each cop sends out a \emph{cop probe}, which gives their distance $d_i$, where $1\le i \le k$, from $u_i$ to the robber's vertex. The distances $d_i$ are nonnegative integers or may be $\infty.$ Hence, in each round, the cops determine a \emph{distance vector} $D=(d_1, d_2, \dots , d_k)$ of cop probes. The cops win if they have a strategy to determine, after a finite number of rounds, the vertex that the robber occupies, at which time we say that the cops {\em capture} the robber. We assume the robber is omniscient, in the sense that they know the entire strategy for the cops. The \emph{localization number} of a graph $G$, written $\loc{G}$, is the least positive integer $k$ for which $k$ cops have a winning strategy.

The minimum number of cops needed to win in the first round (that is, using only one set of cop probes in round 0) is the \emph{metric dimension}, written $\beta(G).$ Observe that $\loc{G} \le \beta(G) \le |V(G)|.$ A survey on metric dimension and related concepts may be found in \cite{bail}, and a recent literature review on the localization number may be found in \cite{BHM}.

The present paper is the first to consider the cop number, localization number, and metric dimension of graphs arising from latin squares. For a positive integer $n,$ a \emph{latin square} of order $n$ is an $n \times n$ array of cells with each cell containing a symbol from a set $S$ with $|S|=n,$ such that each symbol occurs exactly once in each row and in each column.
Often rows are indexed by $R$,  columns are indexed by $C$, and symbols are indexed by $S$. For a latin square $L$, we write its set of \emph{entries} as $$\{(r,c,s) \in R \times C \times S : \text{ symbol $s$ occurs in row $r$ and column $c$ of $L$}\}.$$
We will take $R=C=S=[n] = \{1,2, \ldots, n\}$.  We call the elements of $R$ the \emph{row-indices}, of $C$ the \emph{column-indices}, and of $S$ the \emph{symbol-indices}. The elements of $R \cup C \cup S$ will be known as the \emph{indices}.
Define the \emph{row-line} (or more simply, the row) of a row-index $r$ as the subset of $n$ entries of $L$ that contain $r$, and analogously define \emph{column-line} and \emph{symbol-line}. Each of these is called simply a \emph{line}. Given a latin square $L,$ we denote the symbol in row $r$ and column $c$ by $L[r,c].$

The \emph{latin square graph} of a latin square $L$ of order $n,$ written as $G(L)$, is the graph with $n^2$ vertices labeled with the cells of the latin square, where distinct vertices are adjacent if they share a row, column, or symbol.  See Figure~\ref{fig1} for the graph corresponding to the following latin square of order 3:
$$L_3 = \begin{tabular}{|c|c|c|}
\hline 1 & 2 & 3 \\ \hline
2 & 3 & 1 \\ \hline
3 & 1 & 2 \\ \hline
\end{tabular}.
$$
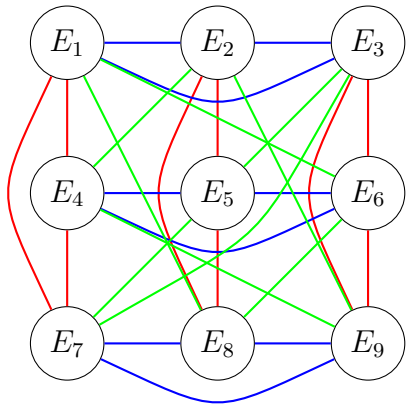
\begin{figure}[h!]
\begin{center}
\begin{tikzpicture}[ node distance={20mm}, main/.style = {draw, circle}]
\node[main] (1) {$E_7$};
\node[main] (2) [above of=1] {$E_4$};
\node[main] (3) [above of=2] {$E_1$};
\node[main] (4) [right of=1] {$E_8$};
\node[main] (5) [above of=4] {$E_5$};
\node[main] (6) [above of=5] {$E_2$};
\node[main] (7) [right of=4] {$E_9$};
\node[main] (8) [above of=7] {$E_6$};
\node[main] (9) [above of=8] {$E_3$};
\draw[color=red, thick] (1) -- (2);
\draw[color=red, thick] (1) to [out=115,in=245, looseness=1.5] (3) ;
\draw[color=red, thick] (2) -- (3);
\draw[color=red, thick] (4) -- (5);
\draw[color=red, thick] (4) to [out=115,in=245, looseness=1.5] (6) ;
\draw[color=red, thick] (5) -- (6);
\draw[color=red, thick] (7) -- (8);
\draw[color=red, thick] (7) to [out=115,in=245, looseness=1.5] (9) ;
\draw[color=red, thick] (8) -- (9);
\draw[color=blue, thick] (1) -- (4);
\draw[color=blue, thick] (1) to [out=335,in=205, looseness=1.5] (7) ;
\draw[color=blue, thick] (4) -- (7);
\draw[color=blue, thick] (2) -- (5);
\draw[color=blue, thick] (2) to [out=335,in=205, looseness=1.5] (8) ;
\draw[color=blue, thick] (5) -- (8);
\draw[color=blue, thick] (3) -- (6);
\draw[color=blue, thick] (3) to [out=335,in=205, looseness=1.5] (9) ;
\draw[color=blue, thick] (6) -- (9);
\draw[color=green, thick] (1) -- (5);
\draw[color=green, thick] (1) to [out=30,in=240, looseness=1.5] (9) ;
\draw[color=green, thick] (5) -- (9);
\draw[color=green, thick] (2) -- (6);
\draw[color=green, thick] (2) --(7); 
\draw[color=green, thick] (6) -- (7);
\draw[color=green, thick] (3) -- (4);
\draw[color=green, thick] (3) --(8); 
\draw[color=green, thick] (4) -- (8);
\end{tikzpicture}
\caption{The graph arising from the latin square $L_3,$ where blue edges come from rows, red edges from columns, and green edges from symbols. Cells are labeled $E_i,$ where $1\le i \le 9.$}\label{fig1}
\end{center}
\end{figure}

We may also consider graphs derived from mutually orthogonal latin squares. A pair of latin squares 
$A$ and $B$
of order $n$ are \emph{orthogonal} if the $n^2$ pairs $(A[i,j],B[i,j])$ 
are distinct. For positive integers $n$ and $k$, a set of $k$ latin squares of order $n$ are \emph{mutually orthogonal}, written $k$-MOLS$(n)$, if the latin squares in the set are pairwise orthogonal. We may write an entry of a $k$-MOLS$(n)$ as $(r,c,s_1,s_2,\ldots ,s_k),$ where $s_i$ is a symbol from the symbol set of the $i$th latin square and $1\le i \le k.$ The maximum number of pairwise orthogonal latin squares is $k=n-1$. The existence of a set of $(n-1)$-MOLS$(n)$ is equivalent to the existence of a (finite) projective plane of order $n$ and an affine plane of order $n$; see \cite{col}. If $\mathcal{L}$ is a set of $k$-MOLS$(n),$ then define the \emph{latin square graph} of $\mathcal{L},$ written $G(\mathcal{L})$, to be the graph with $n^2$ vertices labeled with the cells of the latin square, where distinct vertices are adjacent if the corresponding cells in the latin square share a row, a column, or a symbol from any of the $k$ symbol sets. In the case $k=1,$ these are the latin square graphs. The graph $G(\mathcal{L})$ is $(k+2)(n-1)$-regular.

The cop number of graphs arising from combinatorial designs was studied in \cite{bb}, where bounds and exact values were determined for incidence graphs of designs, polarity graphs, block intersection graphs, and point graphs. That study was partially motivated by the search for new examples of so-called Meyniel extremal families of graphs, which have the conjectured largest asymptotic value of the cop number for connected graphs; see \cite{bbb}. For a latin square graph of order $n,$ the domination number (which upper bounds the cop number) is bounded between $n/2$ and $n,$ but an exact value is not known; see \cite{PPT}. The localization number and metric dimension of designs were studied in \cite{BHM}, where these parameters were studied for incidence graphs of various balanced incomplete block designs such as projective planes, affine planes, and Steiner systems.

The present paper is organized as follows. In Section~\ref{seccn}, we consider the cop number of latin square graphs arising from $k$-MOLS$(n).$ For many instances of the parameters $k$ and $n,$ including the case $k=1,$ we determine the exact value of the cop number. In the remaining cases, we give bounds on the cop number. The metric dimension of latin square graphs is discussed in the next section, and bounds are presented. In particular, for a latin square $L$ of order $n,$ we derive in Theorem~\ref{mdmain} that $n-\sqrt{\frac{n}{3}+\frac{37}{36}}+ \frac{1}{6} \le \beta(G(L)).$ For the family of back-circulant latin squares, we derive that for $n$ sufficiently large with $2,3,5,7 \nmid n$, $\beta(G(B_n))\leq n-1,$ which proves that the lower bound in Theorem~\ref{mdmain} is close to tight. In Section~\ref{secln}, bounds are provided for the localization number of latin square graphs. By using the existence of certain partial transversals, we derive for a latin square of order $n$ that $\zeta(G(L)) \leq n+O\left(\frac{\log{n}}{\log{\log{n}}}\right)$; see Corollary~\ref{cormain}. Our final section presents several open problems on pursuit-evasion on latin square graphs.

Throughout, all graphs considered are simple, undirected, connected, and finite. For a general reference on graph theory, see~\cite{West}.  For background on latin squares, see \cite{col,tm,stinson}.  Unless otherwise stated, $k$ and $n$ are positive integers.

\section{Cop number of latin squares}\label{seccn}

For latin squares of small orders, the cop number of their graphs may be directly computed. By directly checking, the cop number of a latin square of order $1$ or $2$ is 1, order $3$ is $2,$ and order $4$ is 3. Interestingly, the cop number of latin squares equals 3 for all $n\ge 5,$ as we now demonstrate.

We first consider the cop number in the more general setting of MOLS.
\begin{theorem}\label{lem:CN_LSG_upper}
If $\mathcal{L}$ is a set of $k${\rm-MOLS}$(n)$, then we have that
\[
c(G(\mathcal{L})) \leq  k+2.
\]
\end{theorem}
\begin{proof}
Suppose that $k+2$ cops are at play, which we label as $C_1, C_2, \ldots, C_{k+2}$.  For their first move, the cops occupy arbitrary vertices. Suppose that the robber occupies the vertex $v = (r,c,s_1, \ldots, s_k).$ If there is a cop in row $r,$ column $c,$  or on a vertex containing the symbol $s_i,$ where $1\le i \le k,$ then the robber is captured in the next round. Hence, we assume this does not occur for any of the cops.

Note that every symbol (of all the symbol types) and every column is represented in the row-line of each cop. Hence, a cop can always move to a vertex containing one of $c,s_1,s_2, \ldots, s_k.$ A similar statement holds for other lines and so a cop can always move to a vertex containing $r$.
For $1 \leq i \leq k,$ the cop $C_i$ moves to a vertex containing symbol $s_i.$ The cop $C_{k+1}$ moves to its neighbor in row $r$ and cop $C_{k+2}$ moves to its  neighbor in column $c$.
Thus, each cop shares a unique index with $v$. Since each neighbor of the robber must share an index with its current vertex, if the robber moves, then it will be on a neighboring vertex to the cop that also shares this index. If the robber does not move, then it can be captured by any of the cops.
\end{proof}

If $n$ is sufficiently large compared to $k$, then the upper bound in Theorem~\ref{lem:CN_LSG_upper} has a matching lower bound.

\begin{theorem}\label{thm:CN_LSG_lower1}
Suppose that $n > (k+1)^2$. If $\mathcal{L}$ is a set of $k${\rm-MOLS}$(n)$, then
\[
c(G(\mathcal{L}))= k+2.
\]
\end{theorem}
\begin{proof}
The upper bound follows by Theorem~\ref{lem:CN_LSG_upper}.  For the lower bound, assume that $k+1$ cops are at play.
By the latin property, each cop is on at most one of the lines containing the robber. Independently of how the cops play, on each of turn of the robber, they are on a line that does not contain a cop, since each cop can be on at most one line, and there are $k+2$ lines incident with the robber.
During the next round, the robber will move along this line $\ell,$ and we proceed by showing that line $\ell$ contains a vertex such that the robber can move to this vertex without being captured. Suppose, without loss of generality, that $\ell$ is a row-line.

Note that if there is a vertex with corresponding entry $(r,c,s_1, \ldots, s_k)$, then each row-line has $k+1$ distinct entries containing indices $c,s_1, \ldots, s_k$, by the orthogonal property of the MOLS.
Thus, each cop is incident to exactly $k+1$ entries on $\ell,$ and so at most $(k+1)^2$ entries on this line are adjacent to cops. Since $n>(k+1)^2$, there is at least one entry on $\ell$ that is not adjacent to a cop, and the robber moves to such an entry. By repeating this strategy in subsequent rounds, the robber may avoid capture.
\end{proof}

We have the following immediate corollary in the case $k=1.$

\begin{corollary}
If $L$ is latin square of order $n\ge 5,$ then $c(G(L)) = 3.$
\end{corollary}

In the case that $k$ is close to $n$, a lower bound is provided in our next theorem, although we do not know if it is tight.

\begin{theorem}\label{lem:CN_LSG_lower2}
Suppose that $n \leq (k+1)^2$. If $\mathcal{L}$ is a set of $k${\rm-MOLS}$(n)$, then
\[
c(G(\mathcal{L})) \geq \left\lceil \frac{n}{k+1}\right\rceil.
\]
\end{theorem}
\begin{proof}
Suppose that $ \lceil \frac{n}{k+1}\rceil -1 \leq  k$ cops are at play, where this inequality follows since $n \leq (k+1)^2$. Independently of how the cops play, on each turn of the robber, they are on a line that does not contain a cop, since each cop can be on at most one line containing the robber, and there are $k+2$ lines incident with the robber.
The robber will move along this line $\ell.$ Each cop is incident to exactly $k+1$ entries on $\ell,$ and so there are $(\lceil \frac{n}{k+1}\rceil -1)(k+1) <n$ entries on $\ell$ adjacent to cops. There is at least one entry on $\ell$ that is not adjacent to a cop, and the robber moves to such an entry.
The robber may now employ this strategy in subsequent rounds to avoid capture.
\end{proof}

When $k=n-1$ or $k=n-2$, which are the largest possible values of $k$, the lower bound in Theorem~\ref{lem:CN_LSG_lower2} is tight, showing that the lower bound cannot be improved. We derive these facts in the following lemmas.

\begin{lemma}
 If $\mathcal{L}$ is a set of $(n-1)${\rm-MOLS}$(n)$, then $c(G(\mathcal{L})) =1.$
\end{lemma}
\begin{proof} The graph $G(\mathcal{L})$ is the complete graph, which requires exactly one cop to capture the robber.
\end{proof}

\begin{lemma}
 If $\mathcal{L}$ is a set of $(n-2)${\rm-MOLS}$(n)$, then $c(G(\mathcal{L})) =2.$
\end{lemma}
\begin{proof} The lower bound is given by Theorem~\ref{lem:CN_LSG_lower2}.  To show that two cops are sufficient to capture the robber, we first note that every set $\mathcal{L}$ of $(n-2)$-MOLS$(n)$ has a unique latin square, say $L'$, that can appended to $\mathcal{L}$ to form a set of $(n-1)$-MOLS$(n)$.
If a cop is in row $r$ and column $c$, then they can move to any vertex in $G(\mathcal{L})$ except those vertices in row $r'$ and column $c'$ such that $L'[r,c] = L'[r',c']$.

Suppose we place two cops on row $r$, one in column $c_1$ and the other in column $c_2$. Therefore, if the robber is on a vertex in row $r'$ and column $c'$, and if $L'[r,c_1] = L'[r',c']$, then $L'[r,c_2] \neq L'[r',c']$, and the second cop can capture the robber.
\end{proof}

The upper bound in Theorem~\ref{lem:CN_LSG_upper} is not tight when $k\in \{n-2,n-1\}$, and the lower bound in Theorem~\ref{lem:CN_LSG_lower2} is tight. It is possible that both could be tight for values $n < (k+1)^2$ with $k \notin \{n-2,n-1\}$, as it is possible that there is one latin square that reaches the lower bound, and another latin square of the same order that reaches the upper bound. We note that in the case of graphs from $2$-MOLS$(n),$ our results show that the cop number is $4$ for $n\ge 10.$ Analogous (but omitted) arguments improve this to show that the cop number of graphs from $2$-MOLS$(n)$ is 4 if $n\ge 7.$

\section{Metric dimension of latin squares}\label{secmd}

We begin with general results on the metric dimension of graphs derived from MOLS.

\begin{theorem}
If $\mathcal{L}$ is a set of $k${\rm-MOLS}$(n)$, then
\[
\beta(G(\mathcal{L})) \leq (k+2)(2n-k-2).
\]
\end{theorem}
\begin{proof}
Choose any $k+2$ rows and $k+2$ columns, and place $(k+2)(2n-k-2)$-many cops to fill these rows and columns. As this set has $(k+2)(2n-k-2)$ entries, each cop occupies a distinct vertex.

If the robber is on some entry, then the $k+2$ cops on that row and $k+2$ cops on that column will probe a distance of $1$. The only time when $k+2$ cops on the same row (respectively, column) probe a distance of $1$ is when the robber is on that row (respectively, column), since each vertex not on a line $\ell$ has at most $k+1$ neighbors on $\ell$.
As such, the cops know the row and column that the robber is on, and so know the exact location of the robber in the first round.
\end{proof}

Applying this result in the case for latin squares of order $n$ (with $k=1$) yields an upper bound of $6n-9$, which can be substantially improved.

\begin{theorem} \label{thm:LS_MD_upper_nonInter}
If $L$ is a latin square of order $n$ that contains a set of four entries of the form $\{(r_1,c_1,s_1), (r_1,c_2,s_2), (r_2,c_1,s_2), (r_2,c_2,s_3) \}$, where $s_1,s_2,s_3$ are each distinct and $n \geq 5$, then
\[
\beta(G(L)) \leq 2n-3.
\]
\end{theorem}
\begin{proof}
We place $2n-3$ cops on all the entries of the columns $c_1$ and $c_2$, except $\{(r_1,c_2,s_2), (r_2,c_1,s_2), (r_2,c_2,s_3) \}$.
Hence, exactly one symbol, $s_2$, is not represented among the cops.
If the $n-2 \geq 3$ cops in column $c_2$ probe a distance of $1$ to the robber, then the robber could be on either the entry $ (r_1,c_2,s_2)$ or $(r_2,c_2,s_3)$. In the first case, the cop on entry $(r_1,c_1,s_1)$ probes a distance of $1$ to the robber, as they share the same row, and in the second case this cop probes a distance of $2$, and so the two cop-free vertices of column $c_2$ are distinguishable. If the $n-1 \geq 4$ cops in column $c_1$ probe a distance of $1$ to the robber, then they must be on an entry $ (r_2,c_1,s_2)$. Therefore, the cops can capture the robber if the robber is on columns $c_1$ or $c_2$, so we suppose that this is not the case.

If two cops that share a row both probe a distance of $1$, then the robber must be on that row. Further, if at least one other cop probes a distance of $1,$ then that cop shares its symbol with the robber, and so the cops know the row and symbol of the robber. The robber is then captured using the latin property.
If no other cop probes a distance of one, then the entry of the robber contains symbol $s_2$, and so the cops may capture the robber.

An analogous argument holds in the case that two cops who share a symbol both probe a distance of $1$. Therefore, we are left with the cases where the robber is on row $r_1$ or $r_2$, has symbol $s_2$ or $s_3$, and is not in columns $c_1$ or $c_2$.  We observe that there are no such entries that contain $s_2$, as the only entries in row $r_1$ or $r_2$ with symbol $s_2$ occur in column $c_1$ or $c_2$. There is only one entry containing symbol $s_3$ that satisfy these conditions, since each row contains one entry with symbol $s_3$ and entry $(r_2,c_2,s_3)$ does not satisfy these condition. Therefore, the robber can be captured as they are located on the entry of row $r_1$ and symbol $s_3$, which is uniquely determined by the latin property.
\end{proof}

There are some latin squares that are not covered by Theorem~\ref{thm:LS_MD_upper_nonInter}, such as the Cayley table of addition for $\mathbb{Z}_2^k$ for $k \in \mathbb{N}.$ A slight modification is applicable to all latin squares.

\begin{theorem}\label{thm:LS_MD_upper}
If $L$ is a latin square of order $n \geq 4$, then
\[
\beta(G(L)) \leq 2n-2.
\]
\end{theorem}
\begin{proof}
By Theorem~\ref{thm:LS_MD_upper_nonInter}, all cases follow except the case where it is impossible for a subset of entries $\{(r_1,c_1,s_1), (r_1,c_2,s_2), (r_2,c_1,s_2), (r_2,c_2,s_3) \}$ to exist with $s_1,s_2,s_3$ each being unique.

Suppose that $s_1 = s_3$, and we now play as before, except that we also include an additional cop on entry $(r_1,c_2,s_2)$.
The proof of Theorem~\ref{thm:LS_MD_upper_nonInter} is straightforwardly modified to show that the robber's row and symbol are determined by the cops, and so the robber's exact location is known.
\end{proof}

We present a lower bound for graphs arising from mutually orthogonal latin squares.

\begin{theorem} \label{thm:LSG_lower}
If $\mathcal{L}$ is a set of $k${\rm-MOLS}$(n)$, then
\[
 \beta(G(\mathcal{L})) \geq \frac{2n^2-2}{(k+2)(n-1)+4} .
\]
\end{theorem}
\begin{proof}
Suppose we have $c$-many cops, where $c$ is a positive integer that is to be determined. To successfully capture the robber in one move, only one vertex can have distance two to all cops.
At most $c$ vertices can have distance $1$ to one cop and distance $2$ to all other cops. At most $c$ vertices can have distance $0$ to a cop.
We are therefore left with $n^2 -2c-1$ vertices that may have two or more cops at distance $1$ and the remaining cops at distance $2$.
Each of the $c$ cops is adjacent to at most $(k+2)(n-1)$ of these vertices, which is the regularity of the graph.

Consider the induced subgraph with one part being the vertices that contain cops and the other part being those vertices of distance $1$ to two or more cops. We then delete all edges between vertices in the same part, forming a bipartite graph.

The first part has $c$ vertices, each of degree at most $(k+2)(n-1)$, and the second part has $n^2 -2c-1$ vertices, each of degree at least $2$. Thus, counting the edges coming from vertices in the first part yields at most $c(k+2)(n-1)$, while counting the edges coming from the second part yields at least $2(n^2 -2c-1)$ edges. As both of these are counts of the number of edges in the bipartite graph, it follows that  $2(n^2 -2c-1) \leq c(k+2)(n-1)$, yielding $$\frac{2n^2-2}{(k+2)(n-1)+4} \leq c,$$ which finishes the proof.
\end{proof}

For latin squares, this yields a lower bound of $\frac{2n^2-2}{3n+1} = \frac{2n}{3}-O(1)$ for their metric dimension. We improve this bound as follows.

\begin{theorem} \label{mdmain}
If ${L}$ is a latin square of order $n,$ then
\[ \beta(G({L})) \geq n-\sqrt{\frac{n}{3}+\frac{37}{36}}+ \frac{1}{6}.\]
\end{theorem}
\begin{proof} Let $s$ be a positive integer that will be determined later in the proof.  If we play with $n-s$ cops, then independently of how they are employed, there are at least $s$ rows $\overline{R}$ and $s$ columns $\overline{C}$ whose entries do not contain a cop, and $s$ symbols $\overline{S}$ that do not occur in the entries occupied by cops.
Define $R$, $C$, and $S$ as the set of rows, columns, and symbols, respectively, that do contain cops.
Since the cops uniquely determine the distances of the cells in the subsquare $\overline{R} \times \overline{C}$, there cannot be two entries that contain the same symbol in a cell of $\overline{R} \times \overline{C}$, and also $\overline{R}\times \overline{C}$ may contain at most one symbol in $\overline{S}$.
Therefore, there must be $s^2-1$ cops on vertices whose symbols are in $\overline{R}\times \overline{C}$, which we label as $A_1$.

Each symbol in $\overline{S}$ must occur $s$ times in the rows of $\overline{R}$, and since at most one element of $\overline{R}\times \overline{C}$ may contain a symbol in $\overline{S}$, there are $s^2-1$ entries in $\overline{R}\times C$ containing symbols in $\overline{S}$, and there must be a cop in each of these columns.
Further, such a cop cannot determine the location of the robber if that cop is in $A_1.$ To see this, if we suppose that there was a cop in $A_1$ that was the only cop
with distance $1$ to some entry in $\overline{R}\times \overline{C}$ and the only cop to be distance $1$ to some entry in $\overline{R}\times C$, then these two cells cannot be distinguished by the cops. The robber cannot be caught if they move to one of these two entries in the initial round.

Hence, there must be $s^2-1$ cops, say $A_2,$ not in $A_1$. Each symbol in $\overline{S}$ must occur $s$-many times in the columns of $\overline{C}$, and since at most one element of $\overline{R}\times \overline{C}$ may contain a symbol in $\overline{S}$, there are $s^2-1$ entries in $R\times \overline{C}$ containing symbols in $\overline{S}$, and there must be a cop in each of these columns. Further, such a cop cannot determine the location of the robber if said cop is in $A_1\cup A_2$ by an analogous argument as previously outlined. It follows that there must be $s^2-1$ cops not in $A_1\cup A_2.$

Hence, there are $n - s \geq 3(s^2-1)$ cops, which we solve for $s$ as $s = \frac{-1}{6}+\sqrt{\frac{n}{3}+\frac{37}{36}}$.
\end{proof}

By Theorems~\ref{mdmain} and \ref{thm:LS_MD_upper}, the metric dimension of a latin square graph of order $n$ will have metric dimension between somewhat below $n$ and up to $2n$. Two different latin squares graphs of the same order may have different metric dimension, so it is possible that the upper and lower bounds are tight.
We proceed by showing that the lower bound is close to being tight.

The \emph{back circulant} latin square $B_n$, is defined as $B_n[i,j] = i+j -1 \pmod{n}$, where we write $n$ instead of $0$ to remain consistent with our typical symbol set $[n]$. For example, if $n=5,$ we have that:
\begin{figure}[h!]
\begin{center}
\begin{align*}B_5=
\begin{tabular}{|c|c|c|c|c|}
\hline 1 & 2 & 3 & 4 & 5 \\ \hline
 2 & 3 & 4 & 5 &1\\ \hline
 3 & 4 & 5 & 1 & 2 \\ \hline
 4 & 5 & 1 & 2 & 3\\ \hline
 5 & 1 & 2 & 3 & 4\\ \hline
\end{tabular} .
\end{align*}
\end{center}
\end{figure}

We need a few definitions. A \emph{resolving set} in a graph $G$ is a set of $\beta(G)$ vertices that the cops can play on to win the localization game in one round. Suppose $L$ is a latin square of order $n$. For a nonnegative integer $d,$ a \emph{partial transversal} of deficit $d$ in $L$ is a subset of $n-d$ entries $T \subseteq L$ such that each row, each column, and each symbol is represented at most once among the entries of $T$. A partial transversal of deficit $d=0$ is called a \emph{transversal}.
\begin{lemma}
For $n$ sufficiently large with $2,3,5,7 \nmid n$, we have that $$\beta(G(B_n))\leq n-1.$$
\end{lemma}
\begin{proof}
We begin by providing a resolving set of cardinality $n$, which we will later show can be reduced to one of cardinality $n-1$. Place cops on the entries $\{ (i,n+2-3i, n+1-2i) : i \in [n]\}$.
As $2,3 \nmid n$, note that these entries form a transversal of $L$, and so if the robber is on a vertex that does not contain a cop, then exactly three cops will probe a distance of $1$ to the robber (that is, one for each index type).
If two particular cops probe a distance of $1$, then there are at most six entries of $L$ that the robber may be on. We will show that for each of these six entries, if chosen by the robber, there will be a distinct third cop of distance $1$ associated with that choice.

Suppose that the first two cops are on entries  $C_i=(i,n+3-3i, n+2-2i)$ and $C_j=(j,n+3-3j, n+2-2j),$ where $i \neq j$.
Table~\ref{tab:robberPossible} provides the six entries that the robber may be on. The first row of this table, for example, says that if the robber chose the entry that is in the same row as $C_i$ and the same column as $C_j$, then the robber is on the entry $(i, n+3-3j, n+2-3j+i)$.
Table~\ref{tab:copPossible} then provides the location of the third cop to also probe a distance of 1, given that the robber was on either of the six entries that were possible. Observe that the cop $D_e$ corresponds to the case that the robber was on the entry associated with the $e$th row of  Table \ref{tab:robberPossible}.
For example, if the robber was on $(i, n+3-3j, n+2-3j+i)$, then the cop $C_k$ with $k=(3j-i)/2$ would have distance 1 to the robber.

\begin{table}[h!]
\begin{center}
\begin{tabular}{|c|c|c|c|}
\hline
Row&Column&Symbol\\
\hline
$i$	&$n+3-3j$	&$n+2-3j+i$ \\
$i$	&$n+3-2j-i	$&$n+2-2j $\\
$j$	&$n+3-3i	$&$n+2-3i+j $\\
$3i-2j$	&$n+3-3i$	&$n+2-2j$\\
$j$	&$n+3-2i-j	$&$n+2-2i$\\
$3j-2i$	&$n+3-3j$	&$n+2-2i$\\
\hline
\end{tabular}
\end{center} \caption{The six possible locations of the robber, given that the cops on row $i$ and row $j$ probe a distance of 1 to the robber.} \label{tab:robberPossible}
\end{table}
\begin{table}
\begin{center}

\begin{tabular}{|c|c|c|c|}
\hline
Cop&Row&Column&Symbol\\
\hline
$D_1$&$2^{-1}(3j-i)$	&$n+3-3(2^{-1}(3j-i))$	&$n+2-3j+i$\\
$D_2$&$3^{-1}(2j+i)$	&$n+3-2j-i$	      		&$n+2-2x3^{-1}(2j+i)$\\
$D_3$&$2^{-1}(3i-j)$	&$n+3-3(2^{-1}(3i-j))$	&$n+2-3i+j$\\
$D_4$&$3i-2j$	                &$n+3-3(3i-2j)$  		&$n+2-2(3i-2j)$\\
$D_5$&$3^{-1}(2i+j)$	&$n+3-2i-j	$      			&$n+2-2x3^{-1}(2i+j)$\\
$D_6$&$3j-2i$	                &$n+3-3(3j-2i)$  		&$n+2-2(3j-2i)$\\
\hline
\end{tabular}
\end{center}
\caption{The entries of the six additional cops that will probe a distance of 1 if the robber is on the corresponding locations given in Table~\ref{tab:robberPossible}.} \label{tab:copPossible}
\end{table}

\begin{table}
\begin{center}
\begin{tabular}{|c|c|c|c|c|c|c|c|c|c|c|c|}
\hline
&$D_1$&$D_2$&$D_3$&$D_4$&$D_5$&$D_6$\\
\hline
$D_1$&--&$5i=5j$&$4i=4j$&$7i=7j$&$7i=7j$&$3i=3j$\\ \hline
$D_2$&--&--&$7i=7j$&$7i=7j$&$i=j$&$9i=9j$\\ \hline
$D_3$&--&--&--&$3i=3j$&$5i=5j$&$7i=7j$\\ \hline
$D_4$&--&--&--&--&$7i=7j$&$5i=5j$\\ \hline
$D_5$&--&--&--&--&--&$8i=8j$\\ \hline
\end{tabular}
\end{center}
\caption{The equation (modulo $n$) that results when we assume that two cops in Table~\ref{tab:copPossible} share the same row.}
\label{tab:compCops}
\end{table}

Finally, Table~\ref{tab:compCops} shows the resulting equation if we assume that cop $D_e=D_f$, by equating the rows that $D_e$ and $D_f$ are in.
As  $2,3,5,7 \nmid n$, each of these conditions would imply that $i=j$, giving a contradiction of assumptions, and so each triple of cops that probe a distance of 1 will uniquely determine the location of the robber.

This completes the proof that the $n$ entries chosen form a resolving set. To show that $n-1$ is sufficient, we may remove any one cop from this resolving set. Hence, either $2$ or $3$ cops will probe a distance of $1$ to the robber.
In the case that $2$ cops probe a distance of 1, we know that the removed cop would have probed a distance of $1$ if we had not removed it. We therefore have the same information as if we had used the resolving set of cardinality $n$, and so the robber's location is uniquely determined.
\end{proof}

\section{Localization number of latin squares}\label{secln}

As the metric dimension is an upper bound on the localization number, by Theorem~\ref{thm:LS_MD_upper} we have the following.

\begin{corollary}\label{corl}
If $L$ is a latin square of order $n,$ then
$\zeta(G(L)) \leq 2n-2$.
\end{corollary}

The bound in Corollary~\ref{corl} may be greatly improved, however. For a latin square $L$, a \emph{cover} of $L$ is a set of entries of the latin square such that each row, column, and symbol is represented at least once. The minimum cardinality of a cover of $L$ is denoted by $\mathrm{mc}(L)$.  The following two results demonstrate that using a little more than $n$ cops, the cops may capture the robber.

\begin{theorem}\label{tmc}
For a latin square $L$ of order $n$, we have that
\[
\zeta(G(L)) \leq \mathrm{mc}(L)+54.
\]
\end{theorem}
\begin{proof}
Suppose that in the initial round, the robber chooses to occupy the vertex corresponding to entry $(r,c,s)$. The cops then play on the vertices corresponding to the entries in a minimum cover.
Each line containing the robber will also contain a vertex with a cop, by the definition of a cover. There may be multiple vertices that the cop cannot distinguish via their distances, but we will show that the number of such vertices is at most six.

If three or more cops on a common line each probe a distance of $1$, then the robber must be on this line. If there is a second line with two cops that probe a distance of $1$, then the robber must also be on this second line, and the robber's exact location is known. Suppose that this is not the case, and assume without loss of generality the first line was a row-line, of row $r$. As all three lines that contain the robber must contain a cop, there must be two cops that probe a distance of $1$, say on $(r_1,c_1, s_1)$  and $(r_2,c_2, s_2)$, that are known not to share a row with the robber. Thus, the cops can identify that the robber is on one of the two vertices $(r,c_1,s_2)$ or $(r,c_2,s_1)$.

If there are two lines which each contain two cops, then without loss of generality, we may assume that these are row-and column-line, and write the four such cops' vertices as $(r_1, c_1, s_1)$, $(r_1, c_2, s_2)$, $(r_2, c_3, s_3)$, and $(r_3, c_3, s_4)$, where the $r_i$ and $c_i$ are distinct for all choices of $1\le i \le 3.$
The robber must then be on one of the five vertices $(r_1,c_3,s_5)$, $(r_2,c_1,s_2)$ but requiring $s_2=s_4$, $(r_2,c_2,s_1)$ but requiring $s_1=s_4$, $(r_3,c_1,s_2)$ but requiring $s_2=s_3$, and $(r_3,c_2,s_1)$ but requiring $s_1=s_3$.

The remaining case is when there are two cops $C_1$ and $C_2$ that probe $1$ but that each share no index with any other cop that probed one. The robber must be on the intersection of the three lines containing $C_1$ and three lines containing $C_2$.
There are six such vertices.

We have shown that the robber has been found to reside among at most six vertices. The robber takes its next move, and then the cops probe the same vertices as in the previous round but add $54$ cops, placing an additional three cops on each line that contains one of the six vertices that may have contained a robber.
There will always be some line that contains at least four cops that probe a distance of $1$ to the robber, so the robber must be on this line.
Further, there are at least two cops not on this line that also probe a distance of $1$.
Only two vertices on the line may have a distance 1 to these same two cops, hence the robber must be on one of them.

At this stage of the game, the cops know the robber is on one of two vertices, and that these vertices share an index. Without loss of generality, say they share the same row $r$, and that these vertices are labeled by the entries $(r,c_1,s_1)$ and $(r,c_2,s_2)$.
On the next turn of the cops, they play $n$ cops on the $n$ entries in row $r$. Further, an additional two cops play on each the lines of $c_1$, $c_2$, $s_1$, and $s_2$. Note that these moves employ $n+8$ cops. If the robber moved along a row during its last move, then they are now captured, as a cop will probe the exact vertex that the robber would be on.

We may therefore, assume the robber moved along one of the lines of $c_1$, $c_2$, $s_1$, or $s_2$. Without loss of generality, we may assume that the robber was on $(r,c_1,s_1)$ and moved along the line of $c_1$.
The three cops along the line of $c_1$ will probe a distance of $1$, so the cops knows the robber is on column $c_1$. In addition, another cop that is not on $(r,c_1,s_1)$ but is also on row $r$ probes a distance of $1$, and this cop must share its symbol with the robber.
Hence, the cops know the column and symbol of the robber, and so knows the robbers location exactly. We note that $n+8 \leq \mathrm{mc}(L) +54$ since all covers have cardinality at least $n$. The proof of the upper bound follows.
\end{proof}

We have the following asymptotic bound on the localization number of latin square graphs.

\begin{corollary}\label{cormain}
For a latin square $L$ of order $n$, we have that
\[
\zeta(G(L)) \leq n+O\left(\frac{\log{n}}{\log{\log{n}}}\right) .
\]
\end{corollary}
\begin{proof}
In \cite{KPSY}, it was shown that there exists a partial transversal of cardinality $n-O\big(\frac{\log{n}}{\log{\log{n}}}\big)$, and in
\cite{BMSW} it was shown that a partial transversal of cardinality $n-d$ can be used to construct a cover of cardinality $n+d/2$. Therefore, there exists a cover of cardinality $n+O\big(\frac{\log{n}}{\log{\log{n}}}\big)$ in $L.$ The proof now follows from Theorem~\ref{tmc}.
\end{proof}

We also establish a lower bound on the localization number of MOLS.
\begin{theorem} \label{thm:LocMolsLower}
If $\mathcal{L}$ is a set of $k${\rm-MOLS}$(n)$, then
\[
\zeta(G(\mathcal{L})) \geq \frac{2(n-1)}{k+2}.
\]
\end{theorem}
\begin{proof}
We play the game with $c$ cops, and derive a lower bound on $c$ such that these $c$ cops can capture the robber (with $c$ to be determined later). Suppose that the robber was not located during the cops' last turn, and after its turn, the robber informs the cops that the robber is on the entries of some given row, say the set of $n$ entries $A$. This weakens the strategy for only the robber player, and so reduces the number of cops required to capture the robber. Note that if the cops cannot capture the robber on this turn, independent of which row the robber is located, then the cops will never be able to capture the robber in the standard game. Thus, a lower bound on $c$ such that $c$ cops are required to capture the robber during this single round will be a lower bound on $\zeta(G(\mathcal{L}))$.

\begin{figure}[h!]
\begin{center}
\begin{tikzpicture}
\draw [decorate,
    decoration = {calligraphic brace,
        amplitude=5pt}] (0,0.8) --  (6,0.8)
        node[pos=0.5,above=5pt, black]{$A$};
\draw [decorate,
    decoration = {calligraphic brace,
        amplitude=5pt}] (2,0.1) --  (6,0.1)
        node[pos=0.5,above=5pt, black]{$A'$};
\draw[thick] (0,0)--(6,0);
\draw[thick] (2,0)--(2,-0.5);
\draw[thick] (0,-0.5)--(6,-0.5);
\draw[thick] (0,0)--(0,-1);
\draw[thick] (6,-1)--(6,0);
\draw[thick, dashed] (6,-1)--(6,-1.5);
\draw[thick, dashed] (0,-1.5)--(0,-1);
\draw (.25,-.25) circle (0.1cm);
\draw (.75,-.25) circle (0.1cm);
\draw (1.25,-.25) circle (0.1cm);
\draw (1.75,-.25) circle (0.1cm);

\draw (2.25,-.75) circle (0.1cm);
\draw[very thick, dotted] (2.25,-.75)--(2.25,-.25);
\draw[very thick, dotted] (2.25,-.75)--(2.75,-.25);
\draw[very thick, dotted] (2.25,-.75)--(3.25,-.25);
\draw (3.25,-1.00) circle (0.1cm);
\draw[very thick, dotted] (3.25,-1)--(2.75,-.25);
\draw[very thick, dotted] (3.25,-1)--(3.75,-.25);
\draw[very thick, dotted] (3.25,-1)--(4.25,-.25);
\draw (4.25,-1.25) circle (0.1cm);
\draw[very thick, dotted] (4.25,-1.25)--(3.25,-.25);
\draw[very thick, dotted] (4.25,-1.25)--(3.75,-.25);
\draw[very thick, dotted] (4.25,-1.25)--(4.75,-.25);
\draw (5.25,-1.25) circle (0.1cm);
\draw[very thick, dotted] (5.25,-1.25)--(4.25,-.25);
\draw[very thick, dotted] (5.25,-1.25)--(4.75,-.25);
\draw[very thick, dotted] (5.25,-1.25)--(5.25,-.25);
\end{tikzpicture}
\end{center}
\caption{Eight cops attempting to locate a robber along a single row of entries from a set of $2${\rm-MOLS}$(11)$.}
\label{fig:MOLS_loca}
\end{figure}
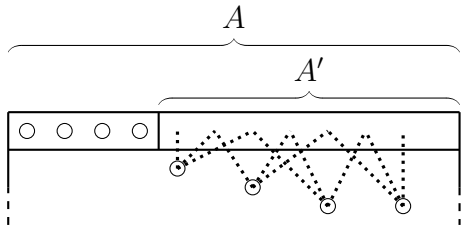

Each cop is either on an entry in $A$, or it is not on $A$ and is adjacent to $k+1$ entries in $A$.
Let $A'$ denote the entries of $A$ that do not contain cops.
Each vertex on $A \setminus A'$ has distance $1$ to each vertex in $A'$, so cannot distinguish which vertex the robber is on if the robber is on a vertex of $A'$.
Let $C'$ denote the set of vertices containing the remaining cops on vertices not on $A$, which have some hope of distinguishing the remaining vertices of $A'$, and let $c' = |C'|$.
See Figure~\ref{fig:MOLS_loca}, which depicts a case with $k=2$, $n=11$, and where eight cops are at play.

Suppose the cops are able to determine the location of the robber on this turn.
Each vertex in $A\setminus A'$ can be immediately localized, as these entries contain a cop, which will probe a distance of $0$.
There can be at most one vertex in $A'$ of distance $2$ to all cops in $C'$.
For each of the cops in $C'$, there can be at most one vertex in $A'$ of distance $1$ to this cop and distance $2$ to all other cops in $C'$.
The most optimal situation for the cops is when each cop is adjacent to exactly one vertex in $A'$ that has the property of being distance  $2$ to all other cops in $C'$, so we assume that this is the case.

We therefore have that $1+c'$ entries in $A'$ have distance $1$ to one or zero cops.
The remaining $|A'|-c'-1$ such vertices must each be adjacent to two cops each.
Label the edges that directly connect these $|A'|-c'-1$ vertices to the cops as $E$.
This means that $E$ contains at least $2(|A'|-c'-1)$ edges.
Each cop is adjacent to at most $k$ such vertices, so $|E| \leq c'k$.
Thus, we must have  $2(|A'|-c'-1) \leq |E| \leq c'k$, and so $$\frac{2(|A'|-1)}{k+2} \leq c'.$$

The total number of cops used is $$c = n-|A'| + c' \geq n - |A'| +  \frac{2(|A'|-1)}{k+2},$$ which is minimized when $|A'|=n$, yielding $c \geq \frac{2(n-1)}{k+2}.$ The proof follows.
\end{proof}

When $k$ is close to $n$, the lower bound in Theorem~\ref{thm:LocMolsLower} does not apply. In certain cases, when $k \geq n/2$, we may substantially improve the lower bound by observing certain properties of the set of MOLS.
An \emph{orthogonal array} OA$(k+2,n)$ is a $(n^2) \times (k+2)$ array, with cells filled with symbols in $[n]$ such that the subarray formed by taking any two columns contain each pair in $[n] \times [n]$ precisely once.
We say that two rows of an orthogonal array \emph{intersect} in a column if both cells of that column in the two rows contain the same symbol.
We note that there is a one-to-one correspondence between a set of $k$-MOLS$(n)$ and an orthogonal array OA$(k+2,n)$; see \cite{stinson}.

\begin{theorem} \label{thm:PairMolsEquiv}
If $\mathcal{M}$ is a set of $k${\rm-MOLS}$(n)$ and $\mathcal{N}$ is a set of $(n-1-k)${\rm-MOLS}$(n)$ such that the composition of the orthogonal arrays of $\mathcal{M}$ and $\mathcal{N}$ is the orthogonal array of a set of $(n-1)${\rm-MOLS}$(n)$, then
\[
\zeta(G(\mathcal{M})) = \zeta(G(\mathcal{N})).
\]
\end{theorem}
\begin{proof}
Let $\mathcal{O}(\mathcal{M})$ and $\mathcal{O}(\mathcal{N})$ denote the orthogonal arrays corresponding to $\mathcal{M}$ and $\mathcal{N}$, respectively.
We write both of these arrays such that the side-by-side composition of the two arrays forms the orthogonal array of a set of $(n-1)${\rm-MOLS}$(n)$, say $\mathcal{O}(\mathcal{L})$.
If a cop $C$ probes a distance of $1$ to the robber $R$ on $\mathcal{M}$, then the rows of $\mathcal{O}(M)$ that correspond to the entries of $R$ and $C$ will intersect, and since the corresponding two rows in $\mathcal{O}(\mathcal{L})$ can only intersect in one column, the two corresponding rows in $\mathcal{O}(\mathcal{N})$ do not intersect.
Similarly, if a cop $C$ probes a distance of $2$ to the robber on $R$ on $\mathcal{M}$, then the rows of $\mathcal{O}(\mathcal{M})$ that correspond to the entries of $R$ and $C$ do not intersect, and since the corresponding two rows in $\mathcal{O}(\mathcal{L})$ do intersect, the two corresponding rows in $\mathcal{O}(\mathcal{N})$ must also intersect.
Equivalent statements hold for $\mathcal{N}$.

We can define a localization game on $\mathcal{O}(\mathcal{M})$ similar to the localization game on graphs, except where the following rules apply.
\begin{enumerate}
\item The cops and robber are placed on rows of the orthogonal array; and
\item The distance between a cop and robber is $0$ if they are on the same row, $1$ if their rows intersect, and $2$ if their rows do not intersect.
\end{enumerate}
By our observations in the first paragraph of this proof, the regular localization game on $G(\mathcal{M})$ is equivalent to playing the new localization game on $\mathcal{O}(\mathcal{M})$.
An equivalent statement holds for $\mathcal{N}$.

By our observations in the first paragraph of this proof, the distance vectors obtained while playing the new localization game on  $\mathcal{O}(\mathcal{M})$ will differ from the distance vectors obtained while playing the new localization game on  $\mathcal{O}(\mathcal{N})$ only in that the $1$'s will be mapped to $2$'s, and vice versa.
Thus, the information that the cops receive is equivalent, independent of whether the game is played on $\mathcal{O}(\mathcal{M})$ or $\mathcal{O}(\mathcal{N})$.
As such, playing the localization game on both $\mathcal{O}(\mathcal{M})$ and $\mathcal{O}(\mathcal{N})$ are equivalent. Since these games were equivalent to the localization game played on $G(\mathcal{M})$ and $G(\mathcal{N})$, we have the desired result that $\zeta(G(\mathcal{M})) = \zeta(G(\mathcal{N}))$.
\end{proof}

By combining Theorems~\ref{thm:LocMolsLower} and \ref{thm:PairMolsEquiv} we derive following result, which is an improvement when $k\geq n/2$. If $i<j,$ then a set $\mathcal{M}$ of $i$-MOLS$(n)$ is \emph{completable} to a set of $j$-MOLS$(n)$ if symbols may be added to $\mathcal{M}$ to form a $j$-MOLS$(n).$

\begin{corollary} \label{cor:MOLS_loc_final}
If $\mathcal{M}$ is a set of $k${\rm-MOLS}$(n)$ that is completable to  a set of $(n-1)${\rm-MOLS}$(n)$, then
\[
\zeta(G(\mathcal{M})) \geq \frac{2(n-1)}{n-k+1}.
\]
\end{corollary}

It is well-known that $(n-1)$-MOLS$(n)$ exist when $n$ is a prime power; see for example, \cite{stinson}. Thus, Corollary~\ref{cor:MOLS_loc_final} shows that when $n$ is a prime power and $k$ is close to $n$, that a set of $k$-MOLS$(n)$ exists such that the localization number is large. In particular, if $k=c$ or $k=n-c,$ where $c$ is a constant, then a set $\mathcal{M}$ of $k$-MOLS$(n)$ exists such that $\zeta(G(\mathcal{M}))=\Theta(n)$.

\section{Future Directions}

We determined the precise cop number of $k$-MOLS$(n)$ when $n > (k+1)^2.$ However, several other cases remain unresolved. For instance, it is unclear whether the bound on the cop number stated in Theorem~\ref{lem:CN_LSG_lower2} is tight. In Sections~\ref{secmd} and~\ref{secln}, for a latin square $L$ of order $n,$ we established the bounds $$n-\sqrt{\frac{n}{3}+\frac{37}{36}}+ \frac{1}{6} \le \beta(G(L)) \le 2n-2,$$ and $$\frac{2}{3}(n-1) \le \zeta(G(L)) \leq n+O\left(\frac{\log{n}}{\log{\log{n}}}\right).$$ We do not know if these bounds are tight.

There are many other graph parameters in pursuit-evasion besides those studied in this paper, such as the $0$-visibility cop number \cite{xue}, the search number \cite{nisses}, and the burning number \cite{abburn}. We will investigate these and other pursuit-evasion parameters on latin square graphs in future work.

\section{Acknowledgements} The second author acknowledges funding from an NSERC Discovery Grant. The first, third, fourth, and six authors conducted research for the paper within the 2021 Fields Undergraduate Summer Research Program. The fifth author was supported by funds from NSERC and The Fields Institute for Research in Mathematical Sciences.

\end{document}